\documentclass{amsart} 
\usepackage{amsmath}
\usepackage[mathscr]{eucal} 
\usepackage{amssymb}
\usepackage{latexsym}
\usepackage{amsthm} 
\theoremstyle{plain} 
\newtheorem*{cor}{Corollary}
\newtheorem{theorem}{Theorem\,\!}[]

\newtheorem{theorema}{Theorem A\!\!}
\newtheorem{theoremb}{Theorem B\!\!}
\newtheorem{theoremc}{Theorem C\!\!}

\newtheorem{lemma}{Lemma}  
 
\newcommand{\sgn}{\mathop{\mathrm{sgn}}\nolimits}

\numberwithin{equation}{section}  
\theoremstyle{definition}

\newtheorem{assertion}{Proposition $P(m)$\!\!}

\theoremstyle{remark}
\newtheorem{remark}{Remark}

\def\XXint#1#2#3{{\setbox0=\hbox{$#1{#2#3}{\int}$}
\vcenter{\hbox{$#2#3$}}\kern-.5\wd0}}

\title[Marcinkiewicz integrals]
{Multiparameter Marcinkiewicz integrals and a resonance theorem }  
\author{Shuichi Sato} 
 
\begin{document} 

\address{Department of Mathematics, 
Faculty of Education, 
Kanazawa University,      
Kanazawa 920-1192, 
Japan}
\email{shuichi@kenroku.kanazawa-u.ac.jp} 
  \thanks{1991 {\it Mathematics Subject Classification.\/}
 42B25. }
\maketitle 

\section{Introduction} \label{s1} 

We shall prove the pointwise relations between some multiparameter square functions on $\bold R^n$, and we shall apply them 
to prove  $H^p-L^p$ and $L(\log L)^{n-1}-L^{1,\infty}$ estimates for 
multiparameter Marcinkiewicz integrals.  To show the weak type estimate,  
we shall also prove a resonance theorem on Orlicz spaces.  

 We first recall one-parameter square functions considered in Sunouchi \cite{21, 22}.  
 For $f \in \mathcal S(\bold R)$  (the Schwartz space), 
  the generalized Marcinkiewicz integral $\mu_{\alpha}(f)$ ($\alpha>0$)  
 is defined by 
 $$\mu_{\alpha}(f)(x) = \left( \int_0^{\infty}|S_t^{\alpha}(f)(x)|^2\,\frac{dt}{t}
 \right)^{1/2},$$ 
 where 
 $$S_t^{\alpha}(f)(x) = \frac{\alpha}{t} \int_0^t\left(1 - \frac{u}{t}\right)^{\alpha - 1}
\delta_u(f)(x)\,du , \quad \delta_u(f)(x) = f(x - u) - f(x + u) .$$ 
For $\alpha>0$, let $\varphi^{(\alpha)}(x)= 
 \alpha|1-|x||^{\alpha -1}(\chi_{[0,1]}(x)-\chi_{[-1,0]}(x))$.  
 Then we note that $S^{\alpha}_t(f)= f\star\varphi^{(\alpha)}_t$, 
 where $\varphi^{(\alpha)}_t(x) = 
 t^{-1}\varphi^{(\alpha)}(t^{-1}x)$. 
  The square function $\mu_1$ coincides with the ordinary Marcinkiewicz integral $\mu$ 
(see \cite{11}, \cite{26}, \cite{24}, \cite{17}) defined by 
$$\mu(f)(x) = \left( \int_0^{\infty}|\Delta_t(\mathcal I_ I(f))(x)|^2t^{-2}\,\frac{dt}{t}
 \right)^{1/2},$$ 
where $\mathcal I(f)(x) = \int_0^xf(y)\,dy$ and
$$\Delta_t(F)(x) = F(x+t) + F(x-t) - 2F(x).$$
\par 
 Let $\hat{f}(t) = \int_{-\infty}^{\infty}f(x)e^{-2\pi i xt}\,dx$ be the Fourier transform. 
 The square functions $h_{\beta}(f)$ ($\beta >0$) and  
 $D_{\alpha}(f)$  ($0<\alpha <1$) are defined as follows : 
$$h_{\beta}(f)(x) = \left( \int_0^{\infty}|\tau_R^{\beta}(f)(x)|^2\,\frac{dR}{R}
 \right)^{1/2},$$
 where
$$\tau_R^{\beta}(f)(x) = \beta\int_{-R}^R\frac{|t|}{R}\left(1 - \frac{|t|}{R}\right)^{\beta - 1}
\hat{f}(t)e^{2\pi i xt}\,dt ; $$
$$ D_{\alpha}(f)(x) =\left(\int_0^{\infty}\frac{|\delta_t(I_{\alpha}(f))(x)|^2}
{t^{1+2\alpha}}\,dt\right)^{1/2},$$
 where
 $$I_{\alpha}(f)(x) = \int_{-\infty}^{\infty} |\xi|^{-\alpha}\hat{f}(\xi)
e^{2\pi ix\xi}\,d\xi $$
is the Riesz potential.  
\par 
Then, the following results are known   
 (see \cite[\S 2]{21},  Theorems 1 and 2 of \cite{22}, and  
also \cite{10} for the $n$-dimensional ($n\geq 2$) case).

\begin{theorema}[\cite{21}] 
If $\hat{f}$ is supported in $[0, \infty)$, 
then $g^*_{2\lambda}(f)(x) \sim h_{\lambda}(f)(x)$  $(\lambda 
> 0)$, that is,  there exist positive constants $A$, $B$ independent of $f$ and $x$ 
 such that
$$Ag^*_{2\lambda}(f)(x) \leq h_{\lambda}(f)(x) \leq Bg^*_{2\lambda}(f)(x).$$
Here, $g^*_{\lambda}(f)$  is the Littlewood-Paley function defined by
$$g^*_{\lambda}(f)(x) = \left(\int_{\bold R^2_+}\left(\frac{t}{t+|x-y|}
\right)^{\lambda}|\nabla u(y,t)|^2 \,dy\,dt\right)^{1/2},$$
where $u(y,t)$ denotes the Poisson integral of $f :$ $u(y,t) = P_t\star f(y)$,  
$P_t(x) = t/\pi(x^2 + t^2)$, and  $\bold R^2_+ = \bold R \times 
(0, \infty)$.
\end{theorema}  

\begin{theoremb}[\cite{22}]  
If $\beta = 1/2 + \alpha$   $(\alpha > 0)$, then
$$h_{\beta}(Hf)(x) \sim \mu_{\alpha}(f)(x),$$
where $H$ denotes the Hilbert transform $: \widehat{Hf}(\xi) = -i\sgn(\xi)
\hat{f}(\xi)$.  
 $($See \cite{4} for the inequality 
$h_{\beta}(Hf)\leq c\mu_{\alpha}(f)$.$)$
\end{theoremb}

\begin{theoremc}[\cite{22}]  If $\beta = 1/2 + \alpha$ and $0<\alpha <1$, then
$$h_{\beta}(Hf)(x) \sim D_{\alpha}(f)(x).$$ 
\end{theoremc}

In this note we consider the multiparameter analogues of the above square functions.  
Let $\alpha = (\alpha_1, \dots , \alpha_n)$, $\alpha_i > 0$;  $\beta = 
(\beta_1, \dots , \beta_n)$, $\beta_i > 0$;  $\lambda = 
(\lambda_1, \dots , \lambda_n)$, $\lambda_i > 0$.  For $f \in 
\mathcal S(\bold R^n)$, we define
$$\mu_{\alpha}(f)(x) = \left( \int_0^{\infty}\dots \int_0^{\infty} 
|S_t^{\alpha}(f)(x)|^2\,\frac{dt_1}{t_1}\dots \frac{dt_n}{t_n}
 \right)^{1/2},$$ 
 where $S^{\alpha}_t = S^{\alpha_1}_{t_1}\otimes \dots \otimes S^{\alpha_n}_{t_n}$ 
 (tensor product),  $t = (t_1, \dots , t_n)$, $x = (x_1, \dots , x_n)$, and so 
 $S_t^{\alpha}(f)(x)= f\star \Phi^{(\alpha)}_t(x)$, with $\Phi^{(\alpha)}_t(x)= 
 \prod_{i=1}^n\varphi^{(\alpha_i)}_{t_i}(x_i)$ ;  
$$h_{\beta}(f)(x) = \left( \int_0^{\infty}\dots \int_0^{\infty} 
|\tau_R^{\beta}(f)(x)|^2\,\frac{dR_1}{R_1}\dots \frac{dR_n}{R_n}
 \right)^{1/2},$$ 
where $\tau^{\beta}_R = \tau^{\beta_1}_{R_1}\otimes \dots \otimes 
\tau^{\beta_n}_{R_n}$,  $R = (R_1, \dots , R_n)$ ; 
$$D_{\alpha}(f)(x) = \left( \int_0^{\infty}\dots \int_0^{\infty} 
|(\delta_{t_1}\otimes \dots \otimes \delta_{t_n})(I_{\alpha}(f))(x)|^2
t_1^{-2\alpha_1}\dots t_n^{-2\alpha_n}\,\frac{dt_1}{t_1}\dots \frac{dt_n}{t_n}
 \right)^{1/2},$$ 
where $I_{\alpha}=I_{\alpha_1}\otimes \dots \otimes I_{\alpha_n}$ and $0<\alpha_i<1$ ;
$$g^*_{\lambda}(f)(x) = \left(\int_{\bold R^2_+\times \dots \times\bold R^2_+}
\prod_{i=1}^n\left(\frac{t_i}{t_i+|x_i - y_i|}\right)^{\lambda_i}
|\nabla_1\dots \nabla_n u(y,t)|^2 \,dy\,dt\right)^{1/2},$$
where $u(y,t)$ denotes the iterated Poisson integral of $f$: $u(y,t) = P_t\star f(y)$,  
$P_t(x) = \prod_{i=1}^nt_i/\pi(x^2_i + t^2_i)$, and  
$$|\nabla_1\dots \nabla_n u(y,t)|^2 = \sum_{i_1=0}^1\dots \sum_{i_n=0}^1 
\left|(\partial/\partial y_{i_1}^{(1)})\dots 
(\partial/\partial y_{i_n}^{(n)}) u(y,t)\right|^2,$$ 
with $\partial/\partial y_{0}^{(j)} = \partial/\partial t_j$, 
$\partial/\partial y_{1}^{(j)} = \partial/\partial y_j$.

We also consider 
$$g_0(f)(x)= \left( \int_0^{\infty}\dots \int_0^{\infty} 
|(\partial/\partial x_1)\dots 
(\partial/\partial x_n) u(x,t)|^2
t_1\dots t_n\, dt_1\dots dt_n \right)^{1/2}.$$ 

Put $\mu = \mu_{(1, \dots, 1)}$.  Then we note that $\mu$ is the multiparameter 
analogue of the ordinary Marcinkiewicz integral on $\bold R$, and it can be 
expressed as follows: 
$$\mu(f)(x) = \left( \int_0^{\infty}\dots \int_0^{\infty} 
|(\Delta_{t_1}\otimes \dots \otimes \Delta_{t_n})(F)(x)|^2t_1^{-2}\dots t_n^{-2}
 \,\frac{dt_1}{t_1}\dots \frac{dt_n}{t_n}\right)^{1/2},$$ 
where $F(x) = (\mathcal I\otimes \dots \otimes \mathcal I)(f)(x) = 
\int_0^{x_1}\dots \int_0^{x_n} f(y_1, \dots , y_n)\,dy_1\dots dy_n$.

Let $f \in \mathcal S (\bold R^n)$.  We shall prove the multiparameter 
analogues of Theorems A, B and C: 

\begin{theorem}  If $\hat {f}$ $($the $n$-dimensional Fourier transform$)$ 
 is supported in $\mathcal D_0 = \{x  
\in \bold R^n : x_i \geq 0 , i = 1, 2, \dots , n \}$, then 
$$g^*_{2\lambda}(f)(x) \sim h_{\lambda}(f)(x) \quad \text{for all} \quad \lambda = 
(\lambda_1, \dots , \lambda_n), \quad \lambda_i > 0. $$ 
\end{theorem}

\begin{theorem}  Let $\alpha = (\alpha_1, \dots , \alpha_n)$,  
$\beta = (\beta_1, \dots , \beta_n)$.   Suppose that $\beta = (1/2, \dots , 1/2) + 
\alpha$, $\alpha_i > 0$.  Then 
$$h_{\beta}((H\otimes \dots \otimes H)(f))(x) \sim \mu_{\alpha}(f)(x).$$
\end{theorem}

\begin{theorem}  Let $\alpha$ and $\beta$ be multi-indices as in Theorem 2.  
 If $\beta = (1/2, \dots , 1/2) + \alpha$,  $0<\alpha_i <1$  $(i= 1, 2, \dots , n)$,  then 
$$h_{\beta}((H\otimes \dots \otimes H)(f))(x) \sim D_{\alpha}(f)(x).$$
\end{theorem}

We shall use  Theorems 1  and 2 to show the $H^p-L^p$ boundedness of $h_\beta$ 
and $\mu_\alpha$.

\begin{theorem}  Let $f \in H^p \cap \mathcal S(\bold R^n)$, where $H^p$ denotes 
the Hardy space on the product domain $\bold R \times \dots \times \bold R$ $($see 
\cite{8}, \cite{9}, and also \cite{14,15,16}$)$.  Then 
\begin{enumerate} 
\item[(1)]  if $1/\beta_i < p < \infty$ and $\beta_i > 1/2$ for $i = 1,2, 
\dots , n$, then 
$$\|h_{\beta}(f)\|_p \leq C_{p, \beta}\|f\|_{H^p};$$
\item[(2)] if $2/(2\alpha_i + 1) < p < \infty$ and $\alpha_i>0$ 
for $i = 1,2, \dots , n$,
 then 
\begin{equation}\tag{a}
\|\mu_{\alpha}(f)\|_p \leq C_{p, \alpha}\|f\|_{H^p} 
\end{equation}  
\begin{equation}\tag{b} 
\|f\|_{H^p} \leq c_{p,\alpha} \|\mu_{\alpha}(f)\|_p.
\end{equation} 
\end{enumerate}
\end{theorem}

In the proof of Theorem 4--(2)--(b) we shall use the following (see \cite{22} 
and \cite{24} for the one-variable case). 

\begin{theorem}
For all $\alpha = (\alpha_1, \dots , \alpha_n)$, $\alpha_i >0$, we have 
$$g_0(f)(x) \leq C_{\alpha}\mu_{\alpha}(f)(x).$$
\end{theorem} 

\begin{remark}
For $\kappa = (\kappa_1, \dots , \kappa_n) \in \{-1, 1\}^n$,  define 
$$E_\kappa = \{x \in \bold R^n : (\kappa_1x_1, \dots , \kappa_nx_n) \in 
\mathcal D_0 \}.$$
Let $\hat{f}_\kappa =\hat{f}\chi_{E_\kappa}$, $f\in \mathcal S$.   
Then, the proof of Theorem 1 below and a change of variables imply that 
 the conclusion of Theorem 1 still holds for $f_\kappa$, for all $\kappa$. 
\end{remark} 

\begin{remark} 
Since  Theorems 2 and 3
imply $\mu_{\alpha}(f) \sim D_{\alpha}(f)$ when $0<\alpha_i <1$, every result 
in this note for $\mu_{\alpha}$, $0<\alpha_i <1$, has an analogue for $D_{\alpha}$ 
(see \cite{22}, \cite{18} for  results  on the one-parameter $D_{\alpha}$). 
\end{remark} 

\begin{remark} 
In the one-parameter case of Theorem 4, we have, in fact, the weak type estimates 
at the endpoints : $p=1/\beta$ (part (1)) and $p=2/(2\alpha +1)$ 
(part (2)--(a)); 
see \cite{1}, \cite{12} and \cite{22}. 
 \end{remark}

 Now, we give the proof of Theorem 4.  
First, we prove part (1).  Let $f_\kappa$ be as in Remark 1. We decompose 
$f=\sum_{\kappa \in \{-1, 1\}^n}f_\kappa$.  Then, by Remark 1, we get 
$$\|h_\beta(f)\|_p\leq c\sum_{\kappa\in \{-1, 1\}^n}\|g^*_{2\beta}(f_\kappa)\|_p.$$
Therefore,  
the conclusion follows from the $H^p-L^p$ boundedness of $g^*_{2\beta}$ 
and the 
inequalities $\|f_\kappa\|_{H^p}\leq c\|f\|_{H^p}$ (see \cite{8}, \cite{9}, 
\cite{2}).  
\par 
Next, we prove part (2)--(a).  
Since $(H\otimes \dots \otimes H)(f_{\kappa})= \kappa_1\dots\kappa_n(-i)^nf_{\kappa}$, 
  using Theorem 2 and arguing as in the proof of part (1), we get 
the conclusion.   
\par   
Finally, part (2)--(b) follows from Theorem 5 since  
$\|f\|_{H^p} \leq c\|g_0(f)\|_p$ 
 (see \cite{8}, \cite{9}, \cite{2}).   This completes the proof of Theorem 4. 
\par 
It is known that the one-parameter operator $\mu_1$ is of weak type $(1,1)$ 
(see \cite{17}). 
We can extend this to the multiparameter case as follows: 

\begin{theorem}   If 
 $|f|(\log(2+|f|))^{n-1}$ is integrable over $\bold R^n$, then for  
$\alpha = (\alpha_1, \dots , \alpha_n)$ with 
 $\alpha_i > 1/2$ $(i=1, 2, \dots , n)$  we have 
$$\left|\left\{x\in \bold R^n : \mu_{\alpha}(f)(x) > \lambda \right\}\right| 
\leq C_{\alpha}\int_{\bold R^n}\frac{|f(x)|}{\lambda}\left(\log 
\left(2 + \frac{|f(x)|}{\lambda}\right)\right)^{n-1}\,dx$$
for all $\lambda > 0$, where $C_{\alpha}$ is independent of $f$ and $\lambda$. 
\end{theorem} 
\par 
See \cite{3} for the weak type estimates of the Calder\'on-Zygmund type
 operators in the two-parameter case.  
 
To prove Theorem 6 we shall require a resonance theorem on the weak type 
boundedness of sublinear operators (Theorem 8 in \S 6) which is a generalization to Orlicz spaces of part of \cite[Chap.\ VI, Corollary 2.9] {6} 
(see also \cite{19}, \cite{13}). 
\par 
We also need to consider vector-valued Marcinkiewicz integrals.  
Let $\mathcal A$ be a separable Hilbert space and let 
$L^p_{\mathcal A}(\bold R^n)$, 
$0<p<\infty$, denote the space of measurable functions $f$ on $\bold R^n$ with 
values in $\mathcal A$ satisfying 
$$\|f\|_{p, \mathcal A}= 
\left ( \int_{\bold R^n}|f(x)|_{\mathcal A}^p\,dx\right)^{1/p}<\infty,$$
where $|\cdot|_{\mathcal A}$ denotes the norm of $\mathcal A$. 
\par 
  Let $f$ be an ${\mathcal A}$-valued measurable function on $\bold R^n$ such 
that  $|f|_{\mathcal A}$ is locally integrable. 
   We define 
$$\mu_{\alpha}^{\mathcal A}(f)(x)
 = \left( \int_0^{\infty}\dots \int_0^{\infty} 
|S_t^{\alpha}(f)(x)|^2_{\mathcal A}\,\frac{dt_1}{t_1}\dots \frac{dt_n}{t_n}
 \right)^{1/2},$$ 
 where $S_t^{\alpha}(f)(x)=f\star \Phi^{(\alpha)}_t(x) = 
 \int f(x-y)\Phi^{(\alpha)}_t(y)\, dy$ 
 is defined by the Bochner integral (see \cite{25}).  
 \par 
 We confine ourselves to the one-parameter case and prove a vector-valued version 
 of Theorem 4--(2)--(a) (for $n=1$). 
 Let $f\in L^1_{\mathcal A}(\bold R)$. We say that 
$f\in H^p_{\mathcal A}(\bold R)$,  $0<p<\infty$,  if 
$$\|f\|_{H^p_{\mathcal A}}= \left ( \int_{-\infty}^{\infty}
\sup_{t>0}|\eta_t\star f(x)|_{\mathcal A}^p\,dx\right)^{1/p}<\infty,$$
where $\eta \in \mathcal S (\bold R)$ with $\int \eta =1$, and 
$\eta_t(x)=t^{-1}\eta(t^{-1}x)$.
 
\begin{cor} 
 If $2/(2\alpha + 1) < p < \infty$,
 then  
 $$\|\mu_{\alpha}^{\mathcal A}(f)\|_p \leq 
C_{p, \alpha}\|f\|_{H^p_{\mathcal A}}
 \qquad \text{for}\quad f\in H^p_{\mathcal A}(\bold R). $$ 
 \end{cor}  

 \begin{proof}[Proof of Corollary]  
 Let $\{e_k\}_{k=1}^\infty$ be an orthonormal basis of $\mathcal A$. If $f$ is 
of the 
 form $f(x) = \sum_{k=1}^Nf_k(x)e_k$, for a positive integer $N$ and functions $f_k\in \mathcal S (\bold R) \cap H^p(\bold R)$, then it is easy to see that 
 the vector-valued analogues of 
 Theorems 1 and 2 hold for $f$.  We also have the vector-valued analogue of 
Remark 1.  
 Using these facts and arguing as in the scalar-valued 
 case above, via the well-known vector-valued theory of Hardy spaces, 
 we get the conclusion of the corollary for such $f$, with the constants 
$C_{p, \alpha}$ 
  independent of $N$.  The general case under consideration   
 follows from a limiting argument.
 \end{proof}  
 
 The proof of Theorem 6 will be given in \S 7 by using the corollary and Theorem 8    
 (we note that Theorem 6 for $n=1$ follows from Theorem 4--(2)--(a) for $n=1$ 
and \cite[Theorem 3.37]{5}).      
 The proofs of Theorems 1, 2, 3, and 5 are similar to those of 
the corresponding results in the one-parameter case (\cite{21}, \cite{22}), 
 but in this note we shall give (essentially) 
 self-contained proofs, for completeness.

\section{Proof of Theorem $1$} 

For $\xi =(\xi_1, \dots , \xi_n)$ put $\Theta_1^{(j)} (\xi) = 2\pi i\xi_j$  
and $\Theta_0^{(j)} (\xi) = -2\pi |\xi_j|$   ($j=1, 2, \dots, n$).  Then 
$$(\partial/\partial y_{i_1}^{(1)})\dots (\partial/\partial y_{i_n}^{(n)}) 
u(y,t) = 
\int \prod_{j=1}^n \left(e^{-2\pi t_j|\xi_j|}\Theta_{i_j}^{(j)} (\xi)\right)
\hat{f}(\xi)e^{2\pi iy\xi} \,d\xi$$
for $i_j = 0, 1$ (see \S 1 for the notation).  
We use the following.

\begin{lemma}  For $\alpha, \,  t>0$ and $x \in \bold R$ we have   
$$\frac{1}{(t- ix)^{\alpha}} = 
\frac{(2\pi)^\alpha}{\Gamma (\alpha)}
\int_0^{\infty}u^{\alpha -1}e^{-2\pi ut}e^{2\pi ixu}\,du,$$
where we consider the principal value of the power in the left hand side.
\end{lemma}

By Lemma 1 we have
\begin{align*} 
&\int (\partial/\partial y_{i_1}^{(1)})\dots (\partial/\partial y_{i_n}^{(n)}) u(y,t) 
\prod_{j=1}^n\frac{1}{\left(t_j - i(y_j - x_j)\right)^{\lambda_j}} e^{-2\pi iy\xi}\,dy 
\\
&= \int \prod_{j=1}^n\left(\frac{(2\pi)^{\lambda_j}}{\Gamma (\lambda_j)}
(\xi_j - \eta_j)_+^{\lambda_j-1}e^{-2\pi|\xi_j-\eta_j|t_j}
e^{-2\pi t_j|\eta_j|}\Theta_{i_j}^{(j)} (\eta)\right)e^{-2\pi ix(\xi-\eta)}\hat{f}(\eta)\,d\eta
\\
&=J , \quad \text{say}, 
\end{align*} 
where $r_+ = \max\{0,r\}$.  Put $c_0 = -2\pi$, $c_1= 2\pi i$.  Then if 
 $\hat{f}$ is supported in $\mathcal  D_0$ and if $\xi \in \mathcal  D_0$, we see that 
\begin{align*} 
&J = e^{-2\pi t\xi}e^{-2\pi i x\xi}\int_0^{\xi_1}\dots 
\int_0^{\xi_n}\prod_{j=1}^n\left(\frac{(2\pi)^{\lambda_j}}{\Gamma (\lambda_j)}
(\xi_j-\eta_j)^{\lambda_j-1}\Theta_{i_j}^{(j)} (\eta)\right)e^{2\pi ix\eta}
\hat{f}(\eta)\,d\eta
\\
&= e^{-2\pi t\xi}e^{-2\pi i x\xi}
\int_0^{\xi_1}\dots \int_0^{\xi_n}\prod_{j=1}^n\left(c_{i_j}\frac{(2\pi)^{\lambda_j}}
{\Gamma (\lambda_j)}\xi_j^{\lambda_j-1}
(1-\eta_j/\xi_j)^{\lambda_j-1}\eta_j\right)e^{2\pi ix\eta}
\hat{f}(\eta)\,d\eta
\\
&=e^{-2\pi t\xi}e^{-2\pi i x\xi}\prod_{j=1}^n
\left(c_{i_j}\frac{(2\pi)^{\lambda_j}}{\Gamma (\lambda_j)}\xi_j^{\lambda_j}
\lambda_j^{-1}\right)\tau^{\lambda}_{\xi}(f)(x).
\end{align*} 
Since $J = 0$ if $\xi \in \bold R^n \setminus \mathcal D_0$, by the Plancherel theorem, we 
have
\begin{multline*}  
\int \left|(\partial/\partial y_{i_1}^{(1)})\dots (\partial/\partial y_{i_n}^{(n)}) 
u(y,t)\right|^2 \prod_{j=1}^n\frac{1}{\left(t_j^2 + |y_j - x_j|^2\right)^{\lambda_j}}\,dy
\\
= (2\pi)^{2n}\int_{\mathcal D_0} e^{-4\pi t\xi}\prod_{j=1}^n
\left(\frac{(2\pi)^{2\lambda_j}}{\Gamma (\lambda_j)^2}\xi_j^{2\lambda_j}
\lambda_j^{-2}\right)\left|\tau^{\lambda}_{\xi}(f)(x)\right|^2\,d\xi.
\end{multline*}    
Multiplying both sides of this equality by 
$t_1^{2\lambda_1}\dots t_n^{2\lambda_n}$ and 
integrating them over $\mathcal D_0$ with respect to the measure $dt$, 
we get the conclusion.

\section{Proof of Theorem $2$} 
Fix $x_0$ and put $\psi(u) = (\delta_{u_1}\otimes \dots \otimes \delta_{u_n})(f)(x_0)$, 
$\Psi(x) = \psi(e^{-x_1}, \dots , e^{-x_n})$. 
By the change of variables $u_j =e^{-y_j}$, $t_j = e^{-x_j}$ we get 
$$\mu_{\alpha}(f)(x_0)^2 
= \int \left|\int_{x_1}^{\infty}\dots \int_{x_n}^{\infty}\left(\prod_{j=1}^n
e^{x_j-y_j}\alpha_j(1-e^{x_j-y_j})^{\alpha_j-1}\right)\Psi(y)\,dy\right|^2\,dx.
$$
Let 
$$K_{\alpha}(x) =  
\begin{cases}  
\prod_{j=1}^n\alpha_je^{x_j}(1-e^{x_j})^{\alpha_j-1}, & \text{if $-x \in 
\mathcal D_0$,} \\
0, & \text{otherwise.}
\end{cases} $$
Then 
\begin{equation}\label{tag3.1} 
\mu_{\alpha}(f)(x_0)^2 
= \int \left|(\Psi\star K_{\alpha})(x)\right|^2\,dx.
\end{equation} 
\par 
Define $F_{\beta}$ by
\begin{align*} 
RF_{\beta}(Ru) &= \beta\int_{-R}^R\frac{|t|}{R}\left(1 - \frac{|t|}{R}\right)^{\beta - 1}
(-i)\sgn (t) e^{2\pi i ut}\,dt
\\ 
&= 2\beta R \int_0^1(1-t)^{\beta -1}t\sin (2\pi Rut)\,dt. 
\end{align*}  
Then  
\begin{multline*} 
h_{\beta}((H\otimes \dots \otimes H)(f))(x_0)^2  
\\
= \int_0^{\infty} \dots \int_0^{\infty}
\left|\int_{D_0} \psi(u) \prod_{j=1}^n R_jF_{\beta_j}(R_ju_j)\,du\right|^2\,
\frac{dR_1}{R_1}\dots \frac{dR_n}{R_n}.
\end{multline*}  
Put $L_{\beta}(x) =\prod_{j=1}^n e^{x_j}F_{\beta_j}(e^{x_j})$.  
By the change of variables $u_j =e^{-y_j}$, $R_j = e^{-x_j}$ we get 
\begin{equation}\label{tag3.2} 
h_{\beta}((H\otimes \dots \otimes H)(f))(x_0)^2  
= \int \left|(\Psi\star L_{\beta})(x)
\right|^2 \,dx.  
\end{equation}

\begin{lemma} For $\alpha >0$ we have   
$$\int_{-\infty}^{0}\alpha e^x (1-e^x)^{\alpha -1}e^{-2\pi ix\xi}\,dx =
\frac{\Gamma(\alpha+1)\Gamma(1-2\pi i\xi)}{\Gamma(\alpha+1-2\pi i\xi)}.$$
\end{lemma} 

\begin{proof} By a change of variables and a well-known relation between 
 the gamma functions and  the Beta functions, we have  
$$ \int_{-\infty}^{0}\alpha e^x (1-e^x)^{\alpha -1}e^{-2\pi ix\xi}\,dx 
=\alpha \int_0^1 t^{-2\pi i \xi}(1-t)^{\alpha -1}\,dt 
= \frac{\Gamma(\alpha+1)\Gamma(1-2\pi i\xi)}{\Gamma(\alpha+1-2\pi i\xi)}. $$
\end{proof} 

\begin{lemma} Let $\beta >0$. If $\varphi \in \mathcal S$ satisfies that 
$e^{x(1-\epsilon)}\varphi (x) \in L^1((0, \infty))$ for all $\epsilon >0$,  
then    
$$\int_{-\infty}^{\infty}e^xF_{\beta}(e^x)\varphi(x)\,dx = 
\int_{-\infty}^{\infty}(2\pi)^{2\pi i \xi +1}
\frac{i\Gamma (\beta +1) \xi}
{2\Gamma (\beta +2\pi i\xi +1)\sin (\pi^2 i\xi)}
\hat{\varphi}(-\xi)\,d\xi. $$
\end{lemma} 

The following estimates for the gamma functions are useful. 
 \begin{lemma}[asymptotic formula for the gamma function] For $x$, 
 $y \in \bold R $, $x>0$,   
$$|\Gamma(x+iy)| \sim \sqrt{2\pi}e^{-\pi|y|/2}|y|^{x-1/2} \qquad 
\text{as \quad $|y| \to \infty$}.$$
\end{lemma}

\begin{proof}[Proof of Lemma 3]   
It is easy to see that 
$$F_{\beta}(u) =  
\begin{cases} 
O\left(|u|^{-\min \{1,\beta \}}\right), & \text{as $|u| \to \infty $ ;} \\
O(|u|), & \text{as $|u| \to 0$.}
\end{cases} $$

Let  $\beta > 1$. Then $e^{(1-\epsilon)x}F_{\beta}(e^x)$ is integrable for  
$\epsilon \in (0, 1)$.  Since 
$$
 \lim_{s \to \infty}\int_0^{s}2\pi u\sin (2\pi ut) t^{-\epsilon -2\pi i\xi}\,dt 
= \Gamma(-\epsilon -2\pi i\xi+1)(2\pi u)^{\epsilon +2\pi i \xi}
\sin \frac{\pi(-\epsilon - 2\pi i \xi +1)}{2}  $$
uniformly in $u\in [0,1]$ (see \cite[p.\ 181]{23}, \cite[p.\ 171]{7}), 
by the change in the order of integration  we see that    
\begin{align*}
&\int_{-\infty}^{\infty} e^{(1-\epsilon)x}F_{\beta}(e^x)e^{-2\pi ix\xi}\,dx =
\int_0^{\infty}F_{\beta}(t)t^{-\epsilon-2\pi i \xi}\,dt
\\ 
&= \int_0^{\infty}\left(2\beta  \int_0^1(1-u)^{\beta -1}u\sin (2\pi ut)\,du\right)
t^{-\epsilon-2\pi i \xi}\,dt 
\\
&=2\beta \Gamma(-\epsilon -2\pi i\xi+1)(2\pi)^{\epsilon +2\pi i \xi -1}
\sin \frac{\pi(-\epsilon - 2\pi i \xi +1)}{2}
\int_0^1(1-u)^{\beta -1} u^{\epsilon +2\pi i \xi}\,du
\\
&=2(2\pi)^{\epsilon +2\pi i \xi -1}
\frac{\Gamma (\beta +1)\Gamma(-\epsilon -2\pi i\xi+1)\Gamma (\epsilon +2\pi i\xi +1)}
{\Gamma (\epsilon +\beta +2\pi i\xi +1)}
\sin \frac{\pi(-\epsilon - 2\pi i \xi +1)}{2}
\\
&= A(\xi,\epsilon), \quad \text{say}.
\end{align*} 

Therefore
\begin{align*}    
&\int_{-\infty}^{\infty}e^xF_{\beta}(e^x)\varphi(x)\,dx = 
\lim_{\epsilon \to 0}\int_{-\infty}^{\infty} e^{(1-\epsilon)x}F_{\beta}(e^x)\varphi(x)\,dx
\\
&=\lim_{\epsilon \to 0} \int_{-\infty}^{\infty}A(\xi,\epsilon) 
\hat{\varphi}(-\xi)\,d\xi
\\
&= \int_{-\infty}^{\infty}2(2\pi)^{2\pi i \xi -1}
\frac{\Gamma (\beta +1)\Gamma(-2\pi i\xi+1)\Gamma (2\pi i\xi +1)}
{\Gamma (\beta +2\pi i\xi +1)}
\cos (-\pi^2 i \xi) \hat{\varphi}(-\xi)\,d\xi
\\
&= \int_{-\infty}^{\infty}(2\pi)^{2\pi i \xi +1}
\frac{\Gamma (\beta +1) i\xi}
{\Gamma (\beta +2\pi i\xi +1)2\sin (\pi^2 i\xi)}
\hat{\varphi}(-\xi)\,d\xi,
\end{align*}   
where we have used well-known formulae for the functions $\Gamma (z)$ and 
$\sin z$.
This proves the conclusion for $\beta > 1$.  Now, by analytic continuation we see that 
 it is valid for all $\beta >0$. This completes the proof of Lemma 3. 
\end{proof}

By (3.1) and the Plancherel theorem we have
\begin{equation}\label{tag3.3} 
\mu_{\alpha}(f)(x_0)^2 = \int \left|\widehat{\Psi}(\xi)
\widehat{K}_{\alpha}(\xi)\right|^2 \,d\xi, 
\end{equation} 
where by lemma 2 
\begin{equation} \label{tag3.4}  
\widehat{K}_{\alpha}(\xi) = \prod_{j=1}^n 
\frac{\Gamma(\alpha_j+1)\Gamma(1-2\pi i\xi_j)}{\Gamma(\alpha_j+1-2\pi i\xi_j)}.
\end{equation} 

Put, for $\beta>0$,  
\begin{equation} \label{3.5} 
A_{\beta}(\xi) = (2\pi)^{2\pi i \xi +1}\frac{i\Gamma (\beta +1) \xi}
{2\Gamma (\beta +2\pi i\xi +1)\sin (\pi^2 i\xi)}.  
\end{equation}  
Note that $\Psi \in \mathcal S(\bold R ^n)$ and 
$(\partial/\partial x_{1})^{\gamma_1}\dots (\partial/\partial x_{n})^{\gamma_n}
\Psi(x) 
= O(\prod_{j=1}^n e^{-|x_j|}) $
as $|x| \to \infty$ for all $\gamma_i \geq 0$. 
 Therefore,  using Fubini's theorem and applying Lemma 3 repeatedly, we have 

$$ (\Psi \star L_{\beta})(x) = \int \widehat{\Psi}(\xi)
\left(\prod_{j=1}^n A_{\beta_j}(\xi_j)\right)e^{2\pi ix\xi}\,d\xi.$$ 
Hence, by (3.2) and the Plancherel theorem, we get 
\begin{equation}\label{tag3.6} 
h_{\beta}((H\otimes \dots \otimes H)(f))(x_0)^2 
= \int \left|\widehat{\Psi}(\xi)
\prod_{j=1}^n A_{\beta_j}(\xi_j)\right|^2 \,d\xi. 
\end{equation} 

By lemma 4, (3.4) and (3.5) we see that 
\begin{equation}\label{tag3.7} 
 \left|\widehat{K}_{\alpha}(\xi)\right| \sim 
\prod_{j=1}^n(1+|\xi_j|)^{-\alpha_j} 
\end{equation} 

\begin{equation}\label{tag3.8}  
\prod_{j=1}^n \left|A_{\beta_j}(\xi_j)\right| \sim 
\prod_{j=1}^n(1+|\xi_j|)^{1/2 -\beta_j}. 
\end{equation}
 By (3.3), (3.6), (3.7) and (3.8) we get the conclusion of Theorem 2. 

\section{Proof of Theorem $3$}

For $0<\alpha <1$ let 
$$k_{\alpha}(x)= 2(2\pi)^{\alpha -1}\Gamma(1-\alpha)\cos ((1-\alpha)\pi/2)
\left(|1-x|^{\alpha -1}-|1+x|^{\alpha -1}\right).$$ 

\begin{lemma}
 Let $f \in \mathcal S (\bold R)$ and $0<\alpha <1$.  Then, 
for $x \in \bold R$ and $t>0$ we have 
$$I_{\alpha}(f)(x-t) -I_{\alpha}(f)(x+t) =t^{\alpha -1} \int_0^{\infty}(f(x-y) -f(x+y))k_{\alpha}(y/t) dy.$$
\end{lemma}

This can be proved by calculating the Fourier transform of $|\xi|^{-\alpha}$ 
(see, for example, \cite[Chap.\ IV]{20}).

\begin{lemma}
 For  $\xi \in \bold R$ we have 
\begin{multline*} 
\int_{-\infty}^{\infty} e^x k_{\alpha}(e^x)e^{-2\pi ix\xi}\,d\xi 
\\
= 4(2\pi)^{\alpha -1}
\Gamma(-2\pi i\xi +1)\sin(-\pi^2i\xi +\pi/2)\Gamma(-\alpha +2\pi i\xi)\sin(\pi^2i\xi 
- \alpha \pi/2). 
\end{multline*}  
\end{lemma} 

\begin{proof} By a change of variables and the formulae in \cite[p.\ 181]{23} 
we have 
\begin{align}\label{tag4.1}   
&\int_{-\infty}^{\infty} e^x k_{\alpha}(e^x)e^{-2\pi ix\xi}\,dx 
= \int_0^{\infty} k_{\alpha}(t) t^{-2\pi i\xi}\,dt 
\\
&= \int_0^{\infty} 4\left( \lim_{S\to \infty}\int_0^{S}\eta^{-\alpha}\sin(2\pi \eta)
\sin(2\pi t\eta)\,d\eta\right) t^{-2\pi i\xi}\,dt     \notag 
\\
&= \lim_{\epsilon\to 0}\lim_{R\to \infty}\int_0^{R} 4\left( \lim_{S\to \infty}
\int_0^{S}\eta^{-\alpha}\sin(2\pi \eta)\sin(2\pi t\eta)\,d\eta\right) 
t^{-2\pi i\xi -\epsilon}\,dt.  \notag 
 \end{align} 
 
Since 
$$\int_0^{S}\eta^{-\alpha}\sin(2\pi \eta)\sin(2\pi t\eta)\,d\eta =
O(|1-t|^{\alpha -1}) \quad \text{uniformly in $S$},$$
by the dominated convergence theorem we get 
\begin{align}\label{tag4.2}  
&\int_0^{R} 4\left( \lim_{S\to \infty}
\int_0^{S}\eta^{-\alpha}\sin(2\pi \eta)\sin(2\pi t\eta)\,d\eta\right) 
t^{-2\pi i\xi -\epsilon}\,dt  
\\
&= \lim_{S\to \infty}4 \int_0^{S}\eta^{-\alpha}
\sin(2\pi \eta) \left( \int_0^{R}\sin(2\pi t\eta)t^{-2\pi i\xi -\epsilon}\,dt
\right) \,d\eta                   \notag
\\
&=4 \int_0^{\infty}\eta^{-\alpha}
\sin(2\pi \eta) \left( \int_0^{R}\sin(2\pi t\eta)t^{-2\pi i\xi -\epsilon}\,dt
\right) \,d\eta                   \notag
\end{align}  
if $\epsilon$ is small enough, where it is to be noted that  
\begin{equation} \label{tag4.3} 
\int_0^{R}\sin(2\pi t\eta)t^{-2\pi i\xi -\epsilon}\,dt = O(\eta^{-1+\epsilon}) 
\quad\text{uniformly in $R$.} 
\end{equation}
Using the relation (4.2) in  (4.1), and applying the dominated convergence 
theorem again, in view of (4.3), we get 
\begin{align*}  
&\int_{-\infty}^{\infty} e^x k_{\alpha}(e^x)e^{-2\pi ix\xi}\,d\xi 
\\
&= \lim_{\epsilon\to 0}\lim_{R\to \infty}4 \int_0^{\infty}\eta^{-\alpha}
\sin(2\pi \eta) \left( \int_0^{R}\sin(2\pi t\eta)t^{-2\pi i\xi -\epsilon}\,dt\right) \,d\eta
\\
&= \lim_{\epsilon\to 0}4 \int_0^{\infty}\eta^{-\alpha}
\sin(2\pi \eta)\lim_{R\to \infty} \left( \int_0^{R}\sin(2\pi t\eta)
t^{-2\pi i\xi -\epsilon}\,dt\right) \,d\eta
\\
&= \lim_{\epsilon\to 0}4\Gamma(-2\pi i\xi +1-\epsilon )\sin(-\pi^2i\xi +\pi/2 -\epsilon \pi/2)
(2\pi)^{2\pi i\xi -1+\epsilon}
\\
&\qquad\qquad \times \int_0^{\infty}\eta^{-\alpha +2\pi i\xi -1+\epsilon}
\sin(2\pi \eta)\,d\eta 
\\
&= \lim_{\epsilon\to 0}4\Gamma(-2\pi i\xi +1-\epsilon )\sin(-\pi^2i\xi +\pi/2 -\epsilon \pi/2)
(2\pi)^{\alpha -1}
\Gamma(-\alpha +2\pi i\xi +\epsilon)
\\
&\qquad \qquad \times \sin(\pi^2i\xi - \alpha \pi/2 +\epsilon\pi/2)
\\
&= 4\Gamma(-2\pi i\xi +1)\sin(-\pi^2i\xi +\pi/2)
(2\pi)^{\alpha -1}\Gamma(-\alpha +2\pi i\xi)\sin(\pi^2i\xi - \alpha \pi/2),
\end{align*}
where we have used the formulae in \cite[p.\ 181]{23} again.  This completes the proof of Lemma 6.
\end{proof} 

Put 
$$J_{\alpha}(x) = \prod_{j=1}^n \left(e^{x_j}k_{\alpha_j}(e^{x_j})\right) \quad (\in L^1(\bold R^n)).$$
Then, by Lemma 5 and a change of variables we see that 
$$ 
D_{\alpha}(f)(x_0)^2 = \int \left|\Psi \star J_{\alpha}(x)\right|^2\,dx = 
\int \left|\widehat{\Psi}(\xi)\widehat{J_{\alpha}}(\xi)\right|^2\,d\xi ,$$
where $\Psi$ is as in \S 3.   By Lemma 4 and Lemma 6  we see that 
$$\left|\widehat{J_{\alpha}}(\xi)\right| \sim 
\prod_{j=1}^n(1+|\xi_j|)^{-\alpha_j}.$$
Comparing this with (3.8), we get the conclusion. 

\section{Proof of Theorem $5$}

We use the same notation as in \S 3.  As in the one-variable case 
(see \cite{22}) 
we see that 
\begin{align*}  
(\partial/\partial x_1)\dots (\partial/\partial x_n) u(x_0,t) &= 
\int_{-\infty}^{\infty}\dots \int_{-\infty}^{\infty} \left(\prod_{j=1}^n 
\frac{-2t_ju_j}{\pi(u_j^2+t_j^2)^2}\right)f(x_0-u)\,du 
\\
&= \int_{0}^{\infty}\dots \int_{0}^{\infty}\left(\prod_{j=1}^n
\frac{-2t_ju_j}{\pi(u_j^2+t_j^2)^2}\right)\psi(u)\,du. 
\end{align*} 
Therefore
$$g_0(f)(x_0)^2= \int_{0}^{\infty}\dots \int_{0}^{\infty}
\left|\int_{0}^{\infty}\dots \int_{0}^{\infty}\left(\prod_{j=1}^n
\frac{-2t_j^2u_j}{\pi(1+t_j^2u_j^2)^2}\right)\psi(u)\,du\right|^2\,
\frac{dt_1}{t_1}\dots\frac{dt_n}{t_n}.$$

Set $K(x)=\pi^{-n}\prod_{j=1}^n(2e^{2x_j}/(1+e^{2x_j})^2)$. Then by the change 
of variables 
$u_j=e^{-y_j}$ and $t_j= e^{x_j}$, we see that 
\begin{equation}\label{tag5.1} 
g_0(f)(x_0)^2= \int\left|\Psi \star K (x) \right|^2\, dx
= \int\left|\widehat{\Psi}(\xi) \widehat{K}(\xi) \right|^2\, d\xi. 
\end{equation} 
As in  \cite{22}, using a well-known one-variable integral formula,  we have 
$$ 
\widehat{K}(\xi)= \prod_{j=1}^n\int_{-\infty}^{\infty}\frac{2e^{(-2\pi ix_j\xi_j+2x_j)}}{\pi(1+e^{2x_j})^2}\,dx_j
= \prod_{j=1}^n\int_{0}^{\infty}\frac{2t_j^{(-2\pi i\xi_j+1)}}
{\pi(1+t_j^{2})^2}\,dt_j
= \prod_{j=1}^n \frac{-\pi i\xi_j}{\sin (-\pi^2 i \xi_j)}.
$$
Thus we see that 
\begin{equation}\label{tag5.2} 
 \left|\widehat{K}(\xi)\right| \sim \prod_{j=1}^n(1+|\xi_j|)e^{-\pi^2|\xi_j|}.  \end{equation} 
By (3.3), (3.7), (5.1) and (5.2) we get the conclusion.

\section{Resonance theorems for sublinear operators on Orlicz spaces}  

Let $\Phi$ be a non-negative, continuous function on $[0, \infty)$ such that
\begin{enumerate}
\item[(1)]  $\Phi$ is increasing on $[0, \infty)$ and $\Phi(0)=0$;
\item[(2)]  $\Phi(t^{1/2})$ is a concave function of $t$ on $[0, \infty)$;
\item[(3)]  there exists a non-negative function 
$b(\lambda)$ ($\lambda \geq 1$) 
such that $\lim_{\lambda \to \infty}b(\lambda)=0$ and $\Phi(\lambda^{-1}t) 
\leq b(\lambda)\Phi(t)$ 
($\lambda \geq 1$);  
\item[(4)]  $\lim_{t \to \infty}\Phi(t)=\infty$. 
\end{enumerate} 

\begin{remark}
It is easy to see that the conditions (1) and (2) imply that 
$\Phi(2t) \leq 4\Phi(t)$ (see \cite{19}).  
 Functions $\Phi(t)=t(\log (2+t))^a$, $a\geq 0$, and $\Phi(t)=t^p$, $0<p\leq 2$, satisfy the  
 conditions (1) to (4).  If $\Phi$ is convex and $\Phi(0)=0$, then $\Phi(\lambda^{-1}t)\leq 
 \lambda^{-1}\Phi(t)$  ($\lambda \geq 1$), and so the condition (3) is 
satisfied.  
\end{remark} 

Let $(X, \mu)$ and $(Y, \nu)$ be measure spaces.  We assume that 
the space $(Y, \nu)$ is $\sigma$-finite.  
Let $L^\Phi(X)$ be the family (Orlicz space) of 
measurable functions $f$ on $X$ such that $\|f\|_{\Phi} =\int_X \Phi(|f|)\,d\mu < \infty$. 
 We note that if $f$,  $g \in L^\Phi(X)$, then $f+g \in L^\Phi(X)$ and $\lambda f \in L^\Phi(X)$ 
 for any scalar $\lambda$. If $\Phi(t)=t(\log (2+t))^a$ ($a\geq 0$), then 
 the space $L^\Phi(X)$ will be denoted by $L(\log L)^a(X)$.  
  
 Let $T$ be an operator from $L^\Phi(X)$ into the space of measurable functions on $Y$. We assume 
 that $T$ is sublinear, that is, 
 $$|T(f+g)| \leq |T(f)|+|T(g)|, \qquad  |T(\lambda f)| = |\lambda||T(f)| $$
 for all $f$, $g \in L^\Phi(X)$ and all scalars $\lambda$.  Furthermore, we assume that $T$ 
 is continuous in measure in the following sense :
 \newline 
 \indent for any $E$  $(\subset Y)$ of finite measure  
 there exists a non-negative function $C_E(\lambda)$ ($\lambda > 0$)  such that   
 $$\lim_{\lambda \to \infty}C_E(\lambda) =0$$ 
 and 
\begin{multline}\label{tag6.1}  
 \nu\left(\left\{x\in E : |Tf(x)|> \lambda  \right\}\right) 
\leq C_E(\lambda) 
 \quad (\lambda >0)  
\\
 \text{uniformly in $f$ satisfying $\|f\|_{\Phi} = 1$}. 
\end{multline} 
 Then we have the following.

 \begin{theorem} 
 There exists a positive weight $w$ on $Y$ such that 
 $$\int_{\left\{x\in Y : |Tf(x)|> \lambda \right\}}w(x)\,d\nu(x) \leq 
 \int_X\Phi\left(\frac{|f(x)|}{\lambda}\right)\,d\mu(x) \qquad (\lambda>0)$$
 for all $f \in L^{\Phi}(X)$.
 \end{theorem}            
 
 This is a generalization of part of \cite[Chap.\ VI, Corollary 2.7]{6}. 
 
 To prove Theorem 7 we prepare some lemmas.

 \begin{lemma} 
 Let $\{f_j\}$ be a finite sequence of functions in $L^\Phi(X)$.  Set $g_t(x)=
 \sum_jr_j(t)f_j(x)$, where $\{r_j\}$ denotes the Rademacher functions.  Then 
 $$\nu\left(\left\{x\in E : \sup_j|Tf_j(x)|> \lambda \right\}\right)\leq 
 2\int_E \left|\left\{t\in [0,1] : |Tg_t(x)|>\lambda \right\}\right|\,d\nu(x) 
 \quad (\lambda>0)$$
 for all measurable subsets $E$ of $Y$.
 \end{lemma} 
 For the proof see \cite[p.\ 535]{6}.
 
 \begin{lemma} Let $\{r_j\}$ be a finite sequence of the Rademacher functions. 
  Set $F(t)=\sum_{j} c_jr_j(t)$. Then 
 $$\int_0^1\Phi(|F(t)|)\,dt \leq \sum_{j}\Phi(|c_j|).$$
 \end{lemma}
 
 See  \cite[ p.\ 155]{19}.

 \begin{lemma} 
 Let $E$ be a subset of $Y$ of finite measure.  Then there exists a non-negative function  
 $\widetilde{C}_E(\lambda)$ $(\lambda>0)$ tending to $0$ as $\lambda \to \infty$ for which 
 the following holds $:$   
 if $\{f_j\}$ is a sequence of functions in $L^\Phi(X)$ satisfying 
 $\sum_j\int_X \Phi(|f_j|)\,d\mu \leq 1$, then 
\begin{equation} \label{tag6.2} 
 \nu\left(\left\{x\in E : \sup_j |Tf_j(x)|> \lambda \right\}\right)\leq 
 \widetilde{C}_E(\lambda) \qquad (\lambda>0).  
\end{equation} 
 \end{lemma} 
 
 \begin{proof}  It is sufficient to prove the lemma for finite sequences 
$\{f_j\}_{j=1}^N$ 
 with $\widetilde{C}_E(\lambda)$ independent of $N$. 
 For $\lambda >1$  put
 $$B_{\lambda}= \left\{t\in [0, 1] : \int_X \Phi(\lambda^{-1/2}|g_t (x)|)
\,d\mu(x)  >1 \right\},$$ 
 where $g_t(x) = \sum r_j(t)f_j(x)$.  Then by Lemma 8 
 \begin{align*}  
 |B_{\lambda}| &\leq \int_0^1 \int_X \Phi(\lambda^{-1/2}|g_t (x)|)\,d\mu(x) 
\, dt
 \\
 &\leq \sum_j \int_X \Phi(\lambda^{-1/2}|f_j (x)|)\,d\mu(x) 
\leq b(\lambda^{1/2}), 
 \end{align*}  
 where we have used part (3) of the properties of $\Phi$ and our assumption 
that  $\sum \|f_j\|_\Phi \leq 1$ to get the last inequality. 

 For $t \in [0, 1]\setminus B_\lambda$, define 
 $$\lambda_t= \inf\left\{\eta >0 : \int_X \Phi(\eta^{-1/2}|g_t (x)|)\,d\mu(x) 
 \leq 1 \right\}.$$
 Then, if $g_t\not= 0$, we have $0< \lambda_t \leq \lambda$ and 
 $\int_X \Phi(\lambda^{-1/2}_t|g_t|)\,d\mu =1$.  
   Using Lemma 7 and these facts, we see that 
\begin{align*}  
&\nu\left(\left\{x\in E : \sup_j|Tf_j(x)|> \lambda \right\}\right) \leq 
 2\int_E \left|\left\{t\in [0,1] : |Tg_t(x)|>\lambda \right\}\right|\,d\nu(x)
 \\
 &\leq 2\nu(E)|B_\lambda| + 2\int_{[0, 1]\setminus B_\lambda}
 \nu\left(\left\{x \in E : |Tg_t(x)| > 
 \lambda^{1/2}\lambda_t^{1/2} \right\}\right) \, dt
 \\
 &\leq 2\nu(E)b(\lambda^{1/2}) + 2C_E(\lambda^{1/2}).
 \end{align*}  
 Therefore we can take $\widetilde{C}_E(\lambda) =  2\nu(E)b(\lambda^{1/2}) + 
 2C_E(\lambda^{1/2})$ for $\lambda >1$ and $\widetilde{C}_E(\lambda) = \nu(E)$ for $\lambda 
 \leq 1$.   This completes the proof of Lemma 9.
  \end{proof} 
 
 \begin{lemma} 
 Let $E$ be a subset of $Y$ of finite measure.  
 For any $\epsilon >0$ there exist a set $E_{\epsilon} \subset E$ and a positive number 
 $R(\epsilon, E)$ such that  $\nu (E\setminus E_{\epsilon})< \epsilon$ and 
 $$\nu\left(\left\{x\in E_{\epsilon} : |Tf(x)|> \lambda \right\}\right)\leq 
 \int_X\Phi\left(\frac{R(\epsilon, E)|f(x)|}{\lambda}\right)\,d\mu(x)
 \qquad (\lambda >0)$$
 for all $f \in L^{\Phi}(X)$. 
 \end{lemma} 
  
 \begin{proof} We assume as we may that $\nu(E) =1$.  
 Given $\epsilon >0$, take $R_\epsilon >0$ such that $\widetilde{C}_E(R_\epsilon) < \epsilon$ 
 (see Lemma 9).  
 Suppose there exists a subset $F $ of $E$ such that 
 \newline 
 \indent ($\dagger$) \quad we can find $f \in L^\Phi(X)$ and $\lambda >0$ for 
which  the following holds : 
\begin{enumerate}
\item[(1)]\qquad  $|Tf(x)| > \lambda$ \qquad for all \qquad $x\in F$;
\item[(2)]  $$\nu(F) > \int_X\Phi\left(\frac{R_{\epsilon}|f(x)|}{\lambda}
\right)\,d\mu(x).$$
\end{enumerate}  
 If there exists no such $F$, then we have the conclusion of Lemma 10. 
 
 Let $\mathcal F$ be the collection of all countable families of disjoint measurable 
 subsets $F$ of $E$ (of positive measure) which satisfy the property ($\dagger$).   
 We define an order $\prec$ in $\mathcal F$ as follows: $\{F_j\} \prec 
 \{G_j\}$  ($\{F_j\}$, $\{G_j\} \in \mathcal F$)  if $\{F_j\} \subset \{G_j\}$.   
 Then, by Zorn's lemma, there exists a maximal 
 element $\{M_j\}$ in $\mathcal F$. 
 
 For each $M_j$ 
 take $f_j \in L^\Phi(X)$ and $\lambda_j>0$ satisfying the relations (1) and (2) with 
 $F=M_j, f=f_j, \lambda=\lambda_j$.  
 Put $S= \cup M_j$ and $g_j=R_\epsilon f_j/\lambda_j$.  
 Then 
 $$\sup_j|Tg_j(x)|>R_\epsilon \qquad \text{for all} \quad x \in S$$
 and 
 $$\sum_j\int_X\Phi(|g_j|)\,d\mu \leq \sum_j\nu(M_j) \leq\nu(E) \leq 1.$$
 So, by  (6.2)
 $$\nu(S) \leq \widetilde{C}_E(R_\epsilon) <\epsilon.$$
 
 Put $E_{\epsilon}=E\setminus S$. We shall see that 
  $$\nu\left(\left\{x\in E_{\epsilon} : |Tf(x)|> \lambda \right\}\right)\leq 
 \int_X\Phi\left(\frac{R_\epsilon |f(x)|}{\lambda}\right)\,d\mu(x),$$
 for all $\lambda >0$ and all $f \in L^\Phi(X)$.  If this does not hold, 
there exist $f_0 \in 
 L^\Phi(X)$ and $\lambda_0>0$ such that 
 $$\nu\left(\left\{x\in E_{\epsilon} : |Tf_0(x)|> \lambda_0 \right\}\right) > 
 \int_X\Phi\left(\frac{R_\epsilon |f_0(x)|}{\lambda_0}\right)\,d\mu(x).$$
 Put $E_0=\{x\in E_{\epsilon} : |Tf_0(x)|>\lambda_0\}$.  Then $E_0$, $f_0$ and $\lambda_0$, 
 in place of  $F$, $f$ and $\lambda$, respectively, satisfy (1) and (2).  
 This contradicts the maximality of $\{M_j\}$. Thus our assertion follows.  This completes 
 the proof of Lemma 10. 
\end{proof}

Now we can prove Theorem 7.  Let $Y = \cup_{j=1}^{\infty}E^{(j)}$, with 
$\nu(E^{(j)}) < \infty$ (the $\sigma$-finiteness).    
Let $R(\epsilon, E^{(j)})>0$ and $E_{\epsilon}^{(j)} \subset E^{(j)}$ be as in 
Lemma 10. 
 Put $R(\epsilon, E^{(j)}) = R(\epsilon, j)$ and $E_{\epsilon}^{(j)} 
=E(\epsilon, j) $. 
  We take $M(\epsilon, j) $ $(\geq 1)$ satisfying 
$$\Phi\left(R(\epsilon, j)t\right) \leq 
M(\epsilon, j)\Phi\left(t\right) \quad (t>0).$$ 

 Define
$$w(x) = \sum_{j=1}^\infty \sum_{m=1}^{\infty}2^{-j}2^{-m}M(1/m, j)^{-1}
\chi_{E(1/m, j)}(x).$$ 
Then $w(x)>0$  $\nu$-a.e.  By Lemma 10  
\begin{align*}  
&\int_{\left\{x\in Y : |Tf(x)|> \lambda \right\}}w(x)\,d\nu(x) 
\\ 
&\leq \sum_{j=1}^\infty \sum_{m=1}^{\infty}2^{-j}2^{-m}M(1/m, j)^{-1}
\nu\left(\left\{x\in E(1/m, j) : |Tf(x)|> \lambda \right\}\right)
\\
&\leq \sum_{j=1}^\infty \sum_{m=1}^{\infty}2^{-j}2^{-m}M(1/m, j)^{-1}
\int_X\Phi\left(\frac{R(1/m, j)|f(x)|}{\lambda}\right)\,d\mu(x)
\\
&\leq \int_X\Phi\left(\frac{|f(x)|}{\lambda}\right)\,d\mu(x).
\end{align*} 
This completes the proof of Theorem 7.  

 We suppose  that $X=Y=\bold R^n$ and $\mu$, $\nu$ are the Lebesgue measure.  
  In addition to the continuity in measure (see (6.1)),  
  we assume that a sublinear operator $T$ is translation-invariant and dilation-invariant 
  on $L^\Phi(\bold R^n)$, which means : 
  $$ T(\tau_h f) = \tau_h Tf \quad \text{for all}\quad h \in \bold R^n 
  \quad \text{and} \quad d_r(Tf)= T(d_r f)  \quad \text{for all}\quad r>0, $$ 
 respectively, where $\tau_h f(x) = f(x-h), d_r f(x) = f(rx)$. 
 Then, using Theorem 7, we can prove the following.

\begin{theorem}  
 There exists a positive constant $A$ such that 
 $$\left|\left\{x\in \bold R^n : |Tf(x)|> \lambda \right\}\right| \leq 
 \int_{\bold R^n}\Phi\left(\frac{A|f(x)|}{\lambda}\right)\,dx 
\quad (\lambda>0)$$
 for all $f\in L^\Phi(\bold R^n)$. 
 \end{theorem}

 The proof is similar to that of Corollary 2.9 of \cite[Chap.\ VI]{6}.

\section{Proof of Theorem $6$}    

We prove Theorem 6 by an argument involving induction on the dimension.  
For this we need to consider  vector-valued Marcinkiewicz integrals.  
Let $m$ be a positive integer. 

\begin{assertion} Let $f$ be a measurable function on 
$\bold R^m$ with values in a separable Hilbert space $\mathcal A$.    
Suppose that $|f|_{\mathcal A}(\log(2+|f|_{\mathcal A}))^{m-1}$ is integrable 
over $\bold R^m$.  
Let $E$ be a compact set in $\bold R^m$.  Then if 
$\alpha = (\alpha_1, \dots , \alpha_m)$ with $\alpha_i > 1/2$, we have 
$$\left|\left\{x\in E : \mu_{\alpha}^{\mathcal A}(f)(x) > \lambda \right\}\right| 
\leq C_{\alpha, E}\lambda^{-1}\left(1 + \int_{\bold R^m}
|f(x)|_{\mathcal A}\left(\log \left(2 + |f(x)|_{\mathcal A}\right)\right)^{m-1}\,dx\right)$$
for all $\lambda > 0$, where $\mu_{\alpha}^{\mathcal A}$ is defined as in 
\S 1 and $C_{\alpha, E}$ 
 is independent of $f$ and $\lambda$.
\end{assertion}

Let $n\geq 2$.  Assuming the proposition $P(m)$ for all $m$ with $1\leq m<n$, we prove $P(n)$.  Let 

\begin{equation} \label{tag7.1} 
\tilde{\mu}_{\alpha}^{\mathcal A}(f)(x) = \left( \int_0^{1}\dots \int_0^{1} 
|S_t^{\alpha}(f)(x)|_{\mathcal  A}^2\,\frac{dt_1}{t_1}\dots \frac{dt_n}{t_n}
 \right)^{1/2}. 
\end{equation}
 Put $I_1=(0,1)$, $I_2=[1, \infty)$.  We decompose 
\begin{equation} \label{tag7.2} 
 \mu_{\alpha}^{\mathcal A}(f)(x)^2 =  \sum_{(j_1,\dots ,j_n) \in 
\{1,2\}^n} \int_{I_{j_1}\times \dots \times I_{j_n}}
|S_t^{\alpha}(f)(x)|_{\mathcal A}^2\,\frac{dt_1}{t_1}\dots \frac{dt_n}{t_n}. 
\end{equation} 
Let $Q=I_1^m$, $R =I_2^{(n-m)}$  with 
$1\leq m < n$. We first consider 
$$F(x) = \int_{Q\times R}
|S_t^{\alpha}(f)(x)|_{\mathcal A}^2\,\frac{dt_1}{t_1}\dots \frac{dt_n}{t_n}.$$
Let $\mathcal K = L^2_{\mathcal A}(R, dt_{m+1}/t_{m+1}\dots dt_n/t_n)$ be the space of  
measurable functions $g$ on $R$ with values in $\mathcal A$ satisfying 
$$|g|_{\mathcal K} = \left(\int_{R} |g(t_{m+1},\dots, t_n)|_{\mathcal A}^2
\frac{dt_{m+1}}{t_{m+1}}\dots 
\frac{dt_n}{t_n}\right)^{1/2} < \infty .$$ 
We write $t'=(t_1,\dots, t_m), t^{\prime \prime}=(t_{m+1},\dots, t_n)$, etc. 
Let $k : \bold R^n \to \mathcal K$ 
be defined a.e. by 
$$(k(x))(t^{\prime \prime}) = \left(S^{\alpha_{m+1}}_{t_{m+1}}\otimes\dots \otimes  
S^{\alpha_{n}}_{t_{n}}\right)(f_{x'})(x^{\prime \prime}) 
\qquad (k(x) : R \to \mathcal A ), $$
where $f_{x'}$ is defined on $\bold R^{n-m}$ by 
 $f_{x'}(x^{\prime \prime}) = f(x)$. 
 Then 
\begin{equation}\label{tag7.3} 
 F(x) = \int_Q \left|\left(S^{\alpha_{1}}_{t_{1}}\otimes\dots \otimes  
S^{\alpha_{m}}_{t_{m}}\right)(k_{x^{\prime \prime}})(x')\right|^2_
{\mathcal K}\frac{dt_1}{t_1}\dots \frac{dt_n}{t_m}, 
\end{equation} 
where $k_{x^{\prime \prime}} : \bold R^{m} \to \mathcal K$ is defined  by 
 $k_{x^{\prime \prime}}(x') = k(x)$ and the operator 
 $S^{\alpha_{1}}_{t_{1}}\otimes\dots \otimes  S^{\alpha_{m}}_{t_{m}}$ 
is acting on the vector-valued function $k_{x^{\prime \prime}}(\cdot)$ ; 
this follows,  of course,  from  the observation that   
 $\left(S^{\alpha_{1}}_{t_{1}}\otimes\dots \otimes S^{\alpha_{m}}_{t_{m}}\right)
 (k_{x^{\prime \prime}})(x') : R \to \mathcal A $ is determined by the relation 
 $$\left((S^{\alpha_{1}}_{t_{1}}\otimes\dots \otimes  
S^{\alpha_{m}}_{t_{m}})(k_{x^{\prime \prime}})(x')\right)(t^{\prime \prime})
=S_t^{\alpha}(f)(x).$$
 
 We recall that $\varphi^{(\alpha)}(u)= 
 \alpha|1-|u||^{\alpha -1}(\chi_{[0,1]}(u)-\chi_{[-1,0]}(u))$ ($\alpha>0$) and note that 
\begin{equation}\label{tag7.4} 
  \int_1^{\infty}|\varphi^{(\alpha)}(u/t)|^2t^{-3}\,dt 
 \leq C_{\alpha}(1+|u|)^{-2}. 
\end{equation} 
 Let $E = \{x \in \bold R^n : |x_i| \leq K \}$.   
 By using the proposition $P(m)$ with $\mathcal A = \mathcal K$   
\begin{multline}\label{tag7.5}  
 \left|\left\{x\in E : F(x)^{1/2} > \lambda \right\}\right|   
\\ 
\leq c\lambda^{-1}\int_{-K}^K\dots \int_{-K}^K M(x_{m+1},\dots ,x_n) 
\,dx_{m+1}\dots dx_n, 
\end{multline} 
where 
$$ M(x_{m+1},\dots ,x_n)=
1+ \int_{\bold R^m}|k(x)|_{\mathcal K}\left(\log 
\left(2 + |k(x)|_{\mathcal K}\right)\right)^{m-1}\,dx_1\dots dx_m.$$
By (7.4) and Minkowski's inequality we see that
$$ |k(x)|_{\mathcal K} \leq c\int 
|f(x',x^{\prime\prime}-y^{\prime\prime})|_{\mathcal A} 
(1+|y_{m+1}|)^{-1}\dots (1+|y_n|)^{-1}\,dy_{m+1}\dots dy_n,  $$
and hence for $\epsilon >0$ by Jensen's inequality we have 
\begin{multline*}  
|k(x)|_{\mathcal K}
\left(\log \left(2 + |k(x)|_{\mathcal K}\right)\right)^{m-1} 
\leq C_{\epsilon}\int |f(x',x^{\prime\prime}-y^{\prime\prime})|_{\mathcal A}
\\ 
\times\left (\log \left (2 + 
|f(x',x^{\prime\prime}-y^{\prime\prime})|_{\mathcal A}\prod_{i=m+1}^n(1+|y_{i}|)^{\epsilon}
\right )\right )^{m-1} 
 \left(\prod_{i=m+1}^n(1+|y_{i}|)^{-1}\right)\,dy^{\prime\prime} 
 \end{multline*}  
 (note that $c_\epsilon\left(\prod_{i=m+1}^n(1+|y_{i}|)^{-1-\epsilon}\right)
 \,dy^{\prime\prime}$ 
 is a probability measure for some $c_\epsilon >0$).
 Thus the right hand side of (7.5) is majorized by 
 $$c\lambda^{-1}+ c\lambda^{-1}\int_{\bold R^n}|f(x)|_{\mathcal A}\left(\log 
\left(2 + |f(x)|_{\mathcal A}\right)\right)^{m-1}\,dx.$$
The other terms except $\tilde{\mu}^{\mathcal A}_{\alpha}(f)$ (see (7.1))  
in the decomposition (7.2) can be treated similarly.

 To estimate $\tilde{\mu}^{\mathcal A}_{\alpha}(f)$ we 
need some preparations.  
 Let $T$ be a subadditive operator 
from $L^1_{\mathcal A}(\bold R)$ (see \S 1 for the definitions of the spaces of vector-valued 
functions) to the space of measurable functions on $\bold R$ ; 
here, the subadditivity  means  
$$|T(f+g)| \leq |T(f)| +|T(g)| \qquad \text{for all} 
\quad f, g \in L^1_{\mathcal A}(\bold R).$$
Then we have the following.

\begin{lemma}
Let $0<p<1<q$ and suppose that 
$$\|Tf\|_p \leq C_p\|f\|_{H^p_{\mathcal A}} \qquad \text{for}\quad  
f\in H^p_{\mathcal A} 
$$
and 
$$\|Tf\|_q \leq C_q\|f\|_{q, \mathcal A} \qquad \text{for} \quad 
f\in L^1_{\mathcal A}\cap L^q_{\mathcal A}.$$ 
Then, for all $\lambda >0$, 
\begin{multline} \label{tag7.6}
|\{x\in \bold R : |Tf(x)|>\lambda \}|
\\ 
\leq 
C_1\lambda^{-q}\int_{|f|_{\mathcal A}\leq \lambda}|f(x)|^q_{\mathcal A}\, dx + C_2\lambda^{-1}\int_{|f|_{\mathcal A}> 
\lambda}|f(x)|_{\mathcal A}\, dx 
\end{multline}  
for $f\in L^1_{\mathcal A}$, with some positive constants $C_1, C_2$. 
\end{lemma} 

In the scalar-valued case, this follows from  \cite[Theorem 3.37]{5}.  
The proof given there easily extends to the vector valued setting under consideration.

\begin{lemma} 
Suppose that $T$ is an operator from $L^1_{\mathcal A}(\bold R)$ to 
the space of 
measurable functions satisfying the estimate $(7.6)$. 
 Then if $|f|_{\mathcal A}\left(\log (2+|f|_{\mathcal A})\right)^{\ell+1}$, 
$0\leq \ell <\infty$, 
 is integrable over $\bold R$,  we have 
 $$\int_{B}|Tf(x)|\left(\log (2+|Tf(x)|)\right)^{\ell}\,dx 
\leq C_{\ell,B}\left(1+ \int_{\bold R}|f(x)|_{\mathcal A}
\left(\log (2+|f(x)|_{\mathcal A})\right)^{\ell+1}\, dx\right),$$ 
 where $B$ is a compact set in $\bold R$.  
\end{lemma} 

\begin{proof}
We  note that if 
$|f|_{\mathcal A}\left(\log (2+|f|_{\mathcal A})\right)^{\ell+1}$ 
is integrable over $\bold R$, then  
$$\lambda\left(\log (2+\lambda)\right)^{\ell}|\{x\in B : |Tf(x)|>\lambda\}| 
\to 0 \quad \text{as $\lambda \to \infty$.}$$ 
To show this, we first see that  
$$ \left(\log (2+\lambda)\right)^{\ell}\int_{|f|_{\mathcal A}> 
\lambda}|f(x)|_{\mathcal A}\,dx 
\leq \int_{|f|_{\mathcal A}> \lambda}|f(x)|_{\mathcal A}
\left(\log (2+|f(x)|_{\mathcal A})\right)^{\ell}\, dx \to 0 $$
as $\lambda \to \infty$.  Next,  
\begin{multline*}  
\lambda^{1-q}\left(\log (2+\lambda)\right)^{\ell}\int_{|f|_{\mathcal A}
\leq \lambda}|f(x)|^q_{\mathcal A}\, dx 
\\
\leq c\left(\log (2+\lambda)\right)^{-1}\int_{|f|_{\mathcal A}\leq \lambda}|f(x)|_{\mathcal A}
\left(\log (2+|f(x)|_{\mathcal A})\right)^{\ell+1}\, dx.
\end{multline*}  
The right hand side also tends to $0$ as $\lambda \to \infty$.  By (7.6), these observations 
imply our assertion. 
 
 Consequently,  we easily see that 
\begin{align} \label{tag7.7}
&\int_B|Tf(x)|\left(\log (2+|Tf(x)|)\right)^\ell\, dx  
\\ 
&= \int_0^\infty\left(\left(\log (2+\lambda)\right)^\ell + 
\frac{\ell\lambda}{\lambda+2}
\left(\log (2+\lambda)\right)^{\ell-1}\right)|\{x\in B : |Tf(x)|>
\lambda\}|\,d\lambda                                  \notag 
\\ 
&\leq c\int_0^\infty \left(\log (2+\lambda)\right)^\ell
|\{x\in B : |Tf(x)|>\lambda\}|\,d\lambda.             \notag
\end{align}   

Note that, for $A>0$, 
$$\int_A^{\infty}\lambda^{-q}(\log (2+\lambda))^\ell\, d\lambda\leq 
C_{\ell,q}A^{-q+1}(\log (2+A))^\ell.$$
Therefore  
\begin{align} \label{tag7.8}  
&\int_0^{\infty}\lambda^{-q}(\log (2+\lambda))^\ell\left(
\int_{|f|_{\mathcal A}\leq \lambda}|f(x)|^q_{\mathcal A}\, dx\right)
\, d\lambda 
\\
&= \int |f(x)|^q_{\mathcal A}\left(\int_{|f(x)|_{\mathcal A}}^{\infty}
\lambda^{-q}(\log (2+\lambda))^\ell\, d\lambda\right)\, dx    \notag 
\\
&\leq c\int |f(x)|_{\mathcal A}\left(\log (2+|f(x)|_{\mathcal A})\right)^\ell\, dx.                                                            \notag
\end{align}  

Next, we note that, for $A>1$, 
$$\int_1^{A}\lambda^{-1}(\log (2+\lambda))^\ell\, d\lambda\leq 
(\log (2+A))^{\ell+1}.$$ 
Using this, we see that    
\begin{align} \label{tag7.9}   
&\int_1^{\infty}\lambda^{-1}(\log (2+\lambda))^\ell\left(
\int_{|f|_{\mathcal A}> \lambda}|f(x)|_{\mathcal A}\, dx\right)\, d\lambda  
\\
&= \int_{|f|_{\mathcal A}> 1} |f(x)|_{\mathcal A}
\left(\int_1^{|f(x)|_{\mathcal A}}
\lambda^{-1}(\log (2+\lambda))^\ell\, d\lambda\right)\, dx      \notag
\\
&\leq \int |f(x)|_{\mathcal A}\left(\log (2+|f(x)|_{\mathcal A})\right)^{\ell+1}\, dx.                                                          \notag
\end{align} 

Since 
$$\int_0^1 \left(\log (2+\lambda)\right)^\ell|\{x\in B : |Tf(x)|>\lambda\}|
\,d\lambda \leq (\log 3)^\ell|B|,$$
by (7.6),  (7.7), (7.8) and (7.9) we get the conclusion.  This completes the proof of 
Lemma 12.  
\end{proof}

By Corollary in \S 1, the sublinear operator $\mu_{\alpha}^{\mathcal A}$ 
on $\bold R$, $\alpha > 1/2$, 
 maps  $H^p_{\mathcal A}(\bold R)$ boundedly into $L^p_{\mathcal A}(\bold R)$ 
 for $p \in (2(2\alpha+1)^{-1}, 1)$. 
Since $\mu_{\alpha}^{\mathcal A}$ is also $L^2_{\mathcal A}$-bounded, 
we can apply Lemma 11 and, consequently, Lemma 12 to $\mu_{\alpha}^{\mathcal A}$.

Let $\mathcal K = L^2_{\mathcal A}((0, 1)^{n-1}, dt_{2}/t_{2}\dots dt_n/t_n)$ be the space of  
measurable functions $g$ on $(0, 1)^{n-1}$ with values in $\mathcal A$ satisfying 
$$|g|_{\mathcal K} = 
\left(\int_0^1\dots \int_0^1 |g(t_{2},\dots, t_n)|_{\mathcal A}^2
\frac{dt_{2}}{t_{2}}\dots 
\frac{dt_n}{t_n}\right)^{1/2} < \infty .$$  
Let $k : \bold R^n \to \mathcal K$ 
be defined a.e., as above,  by 
$$(k(x))(t^{\prime \prime}) = 
\left(S^{\alpha_{2}}_{t_{2}}\otimes\dots \otimes  
S^{\alpha_{n}}_{t_{n}}\right)(f_{x_1})(x^{\prime \prime}), $$
where  $f_{x_1}(x^{\prime \prime}) = f(x)$, 
$x^{\prime \prime}=(x_{2},\dots, x_n)$; $t^{\prime \prime}=(t_{2},\dots, t_n)$. 
 Then, as in (7.3) we can write  
 $$\tilde{\mu}_{\alpha}^{\mathcal A}(f)(x)^2 = 
\int_0^1 \left|S^{\alpha_{1}}_{t_{1}}  
 (k_{x^{\prime \prime}})(x_1)\right|^2_{\mathcal K}\frac{dt_1}{t_1}, $$
where $k_{x^{\prime \prime}} : \bold R \to \mathcal K$ is defined by 
 $k_{x^{\prime \prime}}(x_1) = k(x)$.

 Let $E$ be as in (7.5).  
 We shall prove 
 $$|\{x\in E : \tilde{\mu}_{\alpha}^{\mathcal A}(f)(x) > \lambda\}| \leq c\lambda^{-1} 
 \left(1+ \int |f(x)|_{\mathcal A}(\log (2+ |f(x)|_{\mathcal A}))^{n-1}\, 
dx\right).$$
 To prove this, we assume as we may that $f$ is supported in $E + [-1, 1]^n$. 
 By Corollary and Lemma 11 $\mu_{\alpha_1}^{\mathcal K}$ is of weak type $(1,1)$. 
 Therefore, using Fubini's theorem,   we get 
 $$|\{x\in E : \tilde{\mu}_{\alpha}^{\mathcal A}(f)(x) > \lambda\}| \leq 
 c\lambda^{-1}\int_B |k_{x^{\prime \prime}}(x_1)|_{\mathcal K}\, dx$$
 for some compact set $B$ determined by $E$. 
 Estimating the right hand side by the repeated applications of the similar arguments 
 using Lemma 12 with $\ell=0, 1, \dots , n-2$, we get the desired inequality.

 Collecting the results, we get, for any compact set $E$, 
 $$ 
 |\{x\in E : \mu_{\alpha}^{\mathcal A}(f)(x)>\lambda\}| \leq c\lambda^{-1} +
 c\lambda^{-1}\int_{\bold R^n}|f|_{\mathcal A}\left(\log \left(2+ |f|_{\mathcal A}\right)
 \right)^{n-1}\, dx.$$
 This proves the proposition $P(n)$, under the induction hypothesis.  
 Since the proposition $P(1)$ 
 follows from Corollary and Lemma 11, the proposition $P(n)$ has 
 been proved for all $n$.    
 
 The sublinear operator $\mu_\alpha$ ($\alpha_i>1/2$) is obviously invariant under translation and 
 dilation, and from the proposition $P(n)$ for the scalar-valued functions its continuity 
 in measure on $L(\log L)^{(n-1)}(\bold R^n)$ considered in \S 6 (see (6.1)) follows.  
 Therefore we can apply Theorem 8 with $\Phi(t)= t(\log(2+t))^{(n-1)}$  
 to $\mu_\alpha$ ($\alpha_i>1/2$) to get Theorem 6.

Finally, we remark that the argument used in the proof of Theorem 6 
would also apply to the proofs of weak type estimates for a wider class 
of multiparameter operators.

\end{document}